\documentclass[12pt]{article}

\usepackage{amssymb}
\usepackage{amsthm}
\usepackage{amsmath}
\usepackage{latexsym}
\usepackage{subfigure}
\usepackage{epsfig}
\usepackage{eufrak}
\usepackage[all]{xy}

\title{Computing quaternion quotient graphs via representation of orders}
\author{Luis Arenas-Carmona\thanks{Supported by Fondecyt,
proyecto No. 1120565.}
\\Universidad de Chile,
\\ Facultad de Ciencias,
\\Casilla 653, Santiago,Chile
\\E-mail: learenas@uchile.cl}

\begin{document}

\theoremstyle{plain}
\newtheorem{thm}{Theorem}
\newtheorem{prop}{Proposition}[section]
\newtheorem{prop3}{Proposition}[subsection]
\newtheorem{cor}{Corollary}[prop]
\newtheorem{lem}[prop]{Lemma}
\newtheorem{lem3}[prop]{Lemma}
\theoremstyle{definition}
\newtheorem{ex}[prop]{Example}
\newtheorem{dfn}[prop]{Definition}
\newtheorem{rmk}[prop]{Remark}

\newcommand\classgp{\mathfrak{g}}
\newcommand\anything{*}
\newcommand\sgn{\textnormal{sgn}}
\newcommand\diag{\textnormal{diag}}
\newcommand\disc{\textnormal{disc}}
\newcommand\ad{\mathbb{A}}
\newcommand\Z{\mathbb{Z}}
\newcommand\lu{\mathfrak{L}}
\newcommand\compleji{\mathbb{C}}
\newcommand\alge{\mathfrak{A}}
\newcommand\tr{\textnormal{ tr}}
\newcommand\la{\Lambda}
\newcommand\ideala{\mathcal{A}}
\newcommand\idealb{\mathcal{B}}
\newcommand\bark{\bar{k}}
\newcommand\uno{\{1\}}
\newcommand\quadc{k^*/(k^*)^2}
\newcommand\Ba{\mathfrak B}
\newcommand\oink{\mathcal O}
\newcommand\D{\mathcal D}
\newcommand\nuD{\nu_{\D}}
\newcommand\Q{\mathbb Q}
\newcommand\enteri{\mathbb Z}
\newcommand\gal{\mathcal G}
\newcommand\dete{\textnormal{det}}
\newcommand\ovr[2]{#1|#2}
\newcommand\vaa{\longrightarrow}
\newcommand\Da{\mathfrak{D}}
\newcommand\Ha{\mathfrak{H}}
\newcommand\Ea{\mathfrak{E}}
\newcommand\Ma{\mathfrak{M}}
\newcommand\ma{\mathfrak{m}}
\newcommand\Txi{\lceil}

\newcommand\bmattrix[4]{\left(\begin{array}{cc}#1&#2\\#3&#4\end{array}\right)}
\newcommand\sbmattrix[4]{\textnormal{\scriptsize$\left(\begin{array}{cc}#1&#2\\#3&#4\end{array}\right)$\normalsize}}
\newcommand\blattice[4]{\left<\begin{array}{cc}#1&#2\\#3&#4\end{array}\right>}
\newcommand\bspace[4]{\left[\begin{array}{cc}#1&#2\\#3&#4\end{array}\right]}

\newcommand\wq{\mathfrak{q}}
\newcommand\uink{\mathcal U}
\newcommand\qunitaryomega{\uink}
\newcommand\quniti[4]{\rm{SU}^{#4}}
\newcommand\qunitarylam{\quniti nkh{\la}}

\newcommand\spinomega{Spin_{n}(Q)}
\newcommand\hspinomega{Spin_{n}(\D,h)}

\newcommand\matrici{\mathbb{M}}
\newcommand\finitum{\mathbb{F}}

\newcommand\splitt[1]{\matrici_2(#1)}
\newcommand\splitk{\splitt k}

\newcommand\Hgot{\mathfrak{H}}

\maketitle

\begin{abstract}
We study the correspondence assigning the vertices of a certain quotient of the local Bruhat-Tits tree for $\mathrm{PGL}_2(K)$, where $K$ is a global function field, to conjugacy classes of maximal orders in some quaternion
$K$-algebras. The interplay between quotient graphs and orders can be used to study representation of orders if the
quotient graphs are known and conversely. We use this converse to find a reciprocity law between quotient graph at different places that suffices to compute, recursively, all local quotient graphs when $K$ is rational and the quaternion algebra splits.
\end{abstract}


\section{Introduction}\label{intro}

 In the late seventies, J.-P. Serre and H. Bass showed that the structure of a group $\Gamma$ acting on a tree $T$ can be recovered from the structure of the quotient graph $\Gamma\backslash T$ \cite[Ch.I]{trees}. This theory, now known as Bass-Serre Theory, was used to find generators of certain arithmetic subgroups $\Gamma$ of $\mathrm{PGL}_2(K)$.  \cite[Ch.II]{trees} is mostly concerned with the case
$\Gamma=\mathrm{PGL}_2(A)$ for  the ring $A=A_P$ of functions that are regular outside a single place $P$ of a smooth irreducible curve $X$, with field of constants $K=K(X)$. Using this method, Serre generalized Nagao's Theorem, which expresses $\mathrm{PGL}_2(\finitum[t])$, for any field $\finitum$, as a free product with amalgamation \cite[Ch.II, Th.6]{trees}. He gave the following structural result for these quotient graphs \cite[Ch.II, Th.9]{trees}:
\begin{quote}
\textbf{Theorem S:} The graph  $\Gamma\backslash T_P$, where $T_P$ is the local Bruhat-Tits tree for the group $\mathrm{PGL}_2(K)$ at $P$, is obtained by attaching a finite number of cusps, or infinite half lines, to a certain finite
graph $Y$. The set of such cusps are indexed by the elements in the Picard group $\mathrm{Pic}(X)$.
\end{quote}
Serre also determined the explicit structure of the quotient graph in some specific examples \cite[\S II.2.4]{trees}. 
The proof of Theorem \textbf{S} relies heavily on the fact that the vertices of $T_P$  are in correspondence with certain equivalence classes of vector bundles.
A.W. Mason has given a more elementary proof of these facts  \cite{Mason1}, \cite{Mason2} and applied these graphs to the study of the lowest index non-congruence subgroup of  $\mathrm{PGL}_2(A)$, in a series of joint works with A. Schweitzer \cite{Mason3}, \cite{Mason4}. A few additional quotient graphs are described in
\cite{Mason5} and \cite{takahashi}. M. Papikian has studied the case where $\mathrm{PGL}_2(A)$ is replaced by the group $\mathrm{PGL}_1(D)$, where $D$ is a maximal $A$-order in a quaternion division algebra $\alge$ \cite{Papikian}.

In this article we study family of quotient graphs that classify maximal $X$-orders on a quaternion $K$-algebra $\alge$ splitting at $P$. Since we use the theory of representation fields, we limit ourselves to curves $X$ defined over a finite field $\finitum$. Recall that an $X$-order in $\alge$ is a locally free sheaf of $\oink_X$-algebras whose generic fibre is $\alge$ \cite{abelianos}, \cite{brzezinski87}. These quotient graphs are closely related to the graph $\Gamma\backslash T_P$ studied by Serre, Mason, and Papikian. Let $G=G_P$ be the conjugation stabilizer $G=\mathrm{Stab}_{\alge^*}(D)$, for a maximal $A$-order $D$.  Note that $\Gamma=K^*D^*/K^*$ is a normal subgroup of $G$,
whence the group $G/\Gamma$ acts on $\Gamma\backslash T_P$, and
$G\backslash T_P$ is the quotient graph under this action. We call
$C_P(D)=G\backslash T_P$ the classifying graph, or C-graph of $D$ at $P$, while
$S_P(D)=\Gamma\backslash T_P$ is called the $S$-graph of $D$ in this
work. Note that $\Gamma=PGL_2(A)$ when $D=\matrici_2(A)$.

Let $\Da$ be a maximal $X$-order in $\alge$. Such an order is completely determined by the
completion $\Da_Q$ at every closed place $Q\in X$.
Furthermore, the completion at any finite set of closed places can
be modified to define a new order. In particular, an $A$-order can be extended to an $X$-order by
chosing an arbitrary completion at $P$. An order is maximal if it is
maximal at all places. It follows that the set of maximal orders $\Da$ with a fixed restriction $\Da(U)=D$ to the affine
open subset $U=X\backslash\{P\}$ is in correspondence with the vertices of the local Bruhat-Tits tree $T_P$,
and isomorphism classes of such orders are in correspondence with the vertices of $C_P(D)$. In what follows we write
$C_P(\Da)$ or $S_P(\Da)$ instead of $C_P(D)$ or $S_P(D)$.

Recall that the set $\mathbb{O}$ of maximal $X$-orders in $\alge$ can be split into
spinor genera \cite{abelianos}. There exists an abelian extension $\Sigma/K$ of exponent 2, called the spinor class field,
such that spinor genera can be classified by a distance function $$ \rho:\mathbb{O}\times\mathbb{O}\rightarrow
\mathrm{Gal}(\Sigma/K),$$ i.e., $\Da$ and $\Da'$ are in the same spinor genera if and only if $\rho(\Da,\Da')=\mathrm{Id}_\Sigma$. 
 Spinor genera of $A$-orders are just isomorphism classes when the set of
infinite places of $A$ has strong approximation. This is not the case for $X$-orders. However, spinor genera
still plays an important role in the present setting:

 \begin{thm}\label{th1} In the preceeding notations, the
set of vertices of $C_P(\Da)$ is in
correspondence with the isomorphism classes of two spinor genera of maximal $X$-orders if the Artin symbol $|[P,\Sigma/K]|$ is not trivial on $\Sigma$ and one spinor genera otherwise.
In the former case, each C-graph is bipartite. 
\end{thm}

It follows that the number of
connected graphs that are needed to describe all isomorphism
classes of maximal orders is either $[\Sigma:K]$ or
$[\Sigma:K]/2$. We call the disjoint
union of these graphs the Full  C-graph $C_P=C_P(\alge)$.
The full S-graph $S_P$ is defined analogously.

When $\alge\cong\matrici_2(K)$, by a split maximal order we mean
a conjugate of the sheaf:
$$\Da_B=\bmattrix {\oink_X}{\mathfrak{L}^B}{\mathfrak{L}^{-B}}{\oink_X},$$
where $B$ is an arbitrary divisor on $X$ and $\mathfrak{L}^B$ is
the invertible sheaf defined by
$$\mathfrak{L}^B(U)=\left\{f\in K\Big|\mathrm{div}(f)|_U+B|_U\geq0\right\}.$$
The cusps in Serre's description of the $S$-graph are explicitly described in terms of
 the vector bundles corresponding to
the orders $\Da_B$ for  large enough values of $|\mathrm{deg}(B)|$ \cite[Ch. II]{trees}. In this sense, next theorem is a partial refinement of Serre's result:

 \begin{thm}\label{th2}
The split maximal orders are located in a finite disjoint union of
infinite lines or half-lines. The set of such lines in the Full C-graph is in  correspondence with
the pairs of the form $\{a,-a\}$ in the quotient group $\mathrm{Pic(X)}/\langle\bar P\rangle$, where $\bar P\in\mathrm{Pic(X)}$ denotes the class of $P$. A split order $\Da_B$ is in the line corresponding to $\{\bar{B},-\bar{B}\}$. The half lines correspond to the elements of order $2$.
\end{thm}

In the case of a matrix algebra, the spinor class field of maximal $A$-orders $\Sigma_U$ is the maximal unramified exponent-2 abelian extension splitting $P$, and the Galois group $\mathrm{Gal}(\Sigma_U/K)$ is isomorphic to the maximal exponent-$2$ quotient of  $\mathrm{Pic(X)}/\langle\bar P\rangle$ (\S2). The image of the cusp $\Delta_B$ in \cite[\S II.2.3]{trees} is part of the line containing the order $\Da_{2B}$, so it is always in the trivial component of the C-graph. As follows from Theorem \textbf{S}, the rest of the graph is finite. In \S\ref{vale} we give a general formula for the valencies of v\'ertices in the S-graph, which allows us to compute valencies in the C-graph for all split vertices. When $X=\mathbb{P}_1(\finitum)$ is the projective line, the relation between both graphs can be made explicit as follows:
\begin{thm}\label{equal}
  If $X\cong\mathbb{P}_1(\finitum)$, $\alge\cong\matrici_2(K)$, and $P$ has odd degree, then $C_P$ is isomorphic to $S_P$ and connected. When  $P$ has even degree, there are two connected components in $C_P$ and every vertex of $C_P$ has exactly two pre-images in $S_P$.
\end{thm}
 In particular, when $\mathrm{deg}(P)=1$,  
then $C_P$  is as follows \cite[Ex.II.2.4.1]{trees}:
\[  \xygraph{
!{<0cm,0cm>;<.8cm,0cm>:<0cm,.8cm>::} 
!{(2,0)}*+{\bullet}="a" !{(2.3,0.3)}*+{{}^{\Da_0}}="a1" 
!{(4,0)}*+{\bullet}="b"  !{(4.3,0.3)}*+{{}^{\Da_P}}="b1" 
!{(6,0)}*+{\bullet}="c"  !{(6.3,0.3)}*+{{}^{\Da_{2P}}}="c1"
!{(8,0)}*+{}="d" 
!{(1,1.17)}*+{}="w" 
 "a"-"b" "b"-"c"
"c"-@{.}"d"  } 
\]

The multiplicity $M_P(\Da,\Da')$ of edges joining two particular vertices $\Da$ and $\Da'$ can be explicitely computed, at least for most split vertices, in terms of $N_P(\Da,\Da')$, the number of -neighbors of $\Da$ in $T_P$ that correspond to maximal orders isomorphic to $\Da'$. This is the case for all vertices when $X=\mathbb{P}_1(\finitum)$ and $\alge=\matrici_2(K)$ (cf. \S\ref{multi}). In \S7 we prove the following reciprocity law:

 \begin{thm}\label{th3}
For any pair of maximal orders $(\Da,\Da'')$ and any pair $(P,Q)$ of prime divisors in $X$, we have
$$\sum_{\Da'}N_P(\Da,\Da')N_Q(\Da',\Da'')=\sum_{\Da'}N_Q(\Da,\Da')N_P(\Da',\Da''),$$ where the sum extends
over all isomorphism classes of maximal orders in $\alge$.
\end{thm}

Note that all sums in the theorem are actually finite. In particular, when $X=\mathbb{P}_1(\finitum)$, then $C_P$ can be completely determined by the infinite matrix  $N_P=\Big(N_P(\Da_{iQ},\Da_{jQ})\Big)_{i,j\in\mathbb{N}}$, where $\mathrm{deg}(Q)=1$. When $Q=P$, the matrix is (cf. \S6):
$$N_1:=N_Q=\left(\begin{array}{ccccc}
0&p&0&0&\cdots\\ p+1&0&p&0&\cdots\\ 0&1&0&p&\cdots\\ 0&0&1&0&\cdots\\
\vdots& \vdots& \vdots& \vdots& \ddots\end{array}\right).$$
In this context, all matrices $N_P$ are described by next result:
 \begin{thm}\label{th4}
For any place $P\in\mathbb{P}_1(\finitum)$, the matrix $N_P=N_{\mathrm{deg}(P)}$ depends only onthe degree of
 $P$, and can be computed by the recurrence relation 
\begin{equation}\label{recurr}
N_d=N_1^d-\sum_{i=1}^{[d/2]}{d\choose i}p^iN_{d-2i}.\end{equation}
\end{thm}

\section{Orders and spinor genera}\label{ndos}

Recall that an $X$-order in a $K$-vector space $V$ is a locally free subsheaf of the constant sheaf $V$ \cite{brzezinski87}.
For any sheaf of groups $\Lambda$ on $X$ we let $\Lambda(U)$ denote the group of $U$-sections. In particular,
 $\Lambda(X)$ is the group of global sections. In all that follows, we assume that $\finitum$
is the whole field of constants in $K$, in the sense that
$\oink_X(X)=\finitum$, as otherwise $\finitum$ can be replaced
with a larger field. Let $\alge$ be a central simple $K$-algebra.
In this section we review the basic facts about spinor genera and spinor class fields of orders.
See \cite{abelianos} for details.

Let $|X|$ be the set of closed points in $X$.
Let $\ad=\ad_X$ be the adele ring of $X$, i.e., the subring of $\prod_{P\in|X|}K_P$ of elements that are integral 
at almost all places. Let $\alge_\ad=\alge\otimes_K\ad$ be the adelization of $\alge$.
Both  $\ad$ and $\alge_\ad$ are given the adelic topology \cite[\S IV.1]{weil}. More generally, for any finite
dimensional $K$-vector space $V$, we can define the adelization $V_\ad=V\otimes_K\ad$ endowed with the product topology. For any $\oink_X$-lattice $\Lambda$, the adelization $\Lambda_\ad=\prod_{P\in|X|}\Lambda_P$, is an open and compact subgroup of
$V_\ad$. In particular, the ring of integral ideles $\oink_\ad=(\oink_X)_\ad$ is open and closed in
$\ad$. Furthermore, every open and compact $\oink_\ad$-sub-modules of $V_\ad$ is the adelization of a lattice. For any lattice
$\Lambda$ and any adelic element $a\in\big(\mathrm{End}_K(V)\big)_\ad\cong\mathrm{End}_\ad(V_\ad)$, the lattice $L=a\Lambda$ is the lattice defined by $L_\ad=a\Lambda_\ad$. 

Since any two maximal orders are locally conjugate at all places, if we fix a maximal order $\Da$, any other maximal $X$-order on $\alge$ has the form $\Da'=a\Da a^{-1}$ for some adelic element $a\in\alge_\ad^*$. In a more general theory it is said that
two maximal orders are always in the same genus \cite{Eichler2}.   Two maximal orders $\Da$ and $\Da'$ are in the same spinor genus if $a$ can be chosen of the form $a=bc$ where $b\in\alge$ and $N(c)=1_\ad$, where $N:\alge_\ad^*\rightarrow \ad^*=:J_X$ is the reduced norm on adeles. The spinor class field is defined as the class field corresponding to the set $K^*H(\Da)\subseteq J_X$, where
$$H(\Da)=\{N(a)|a\in\alge_\ad^*,\  a\Da a^{-1}=\Da\}.$$
Let $t\mapsto [t,\Sigma/K]$ denote the Artin map on ideles.
The distance between the maximal orders is the element
$\rho(\Da,\Da')\in\mathrm{Gal}(\Sigma/K)$ defined by $\rho(\Da,\Da')=[N(a),\Sigma/K]$, for any adelic element $a\in\alge_\ad^*$
satisfying $\Da'=a\Da a^{-1}$. Note that this implies that $\rho(\Da,\Da'')=\rho(\Da,\Da')\rho(\Da',\Da'')$ for any triple $(\Da,\Da',\Da'')$ of maximal orders. The spinor class field can be defined also for any affine subset of $X$. In fact,
the spinor class field $\Sigma_U$ corresponding to an affine set $U\subseteq X$ is the largest subfield of $\Sigma$ completely
splitting every place in $S=X\backslash U$.

One important property of spinor genera is that they coincide with conjugacy classes whenever strong approximation holds.
In the context of $X$-orders, this implies that two maximal orders are in the same
spinor genus if and only if they are isomorphic (as sheaves) in
every affine subset $U$ whose complement $S$ has a place
splitting $\alge$. More generally, for a given affine subset $U$ satisfying this condition,
two $S$-orders $\Da(U)$ and $\Da'(U)$ are isomorphic if and only if the distance
$\rho(\Da,\Da')$ is in the group $\left\langle|[P,\Sigma/K]|\Big|P\in S\right\rangle$,
 where $\Sigma$ is the spinor class field of maximal $X$-orders for $\alge$, and $I\mapsto |[I,\Sigma/K]|$ 
is the artin map on ideals  (see \cite[\S2]{abelianos} or \cite[\S2]{ab2}). In all that follows,
we assume $S=\{P\}$ for a fixed place at infinity $P$ splitting $\alge$.

Let $\Ha$ be a suborder of a maximal order $\Da$, and let
$$H(\Da|\Ha)=\{n(a)|a\Ha_\ad a^{-1}\subseteq\Da_\ad,\
a\in\alge^*_\ad\}\subseteq J_X.$$ 
When any of the following equivalent
conditions holds:
\begin{enumerate} \item  the set $K^*H(\Da|\Ha)\subseteq J_K$ is a group,
\item the set
$\Phi=\{\rho(\Da,\Da')|\Ha\subseteq\Da'\}\subseteq\mathrm{Gal}(\Sigma/K)$
is a group,\end{enumerate} then the class field $F(\Ha)$  corresponding to $K^*H(\Da|\Ha)$, or equivalently, the
fixed field $\Sigma^\Phi$, is called the representation field for $\Ha$.
The representation field is not always defined for central simple algebras of arbitrary dimension,
but this is indeed the case for quaternion algebras \cite{Eichler}. When $\alge$ is a quaternion algebra and $\Ha$
is the maximal order in a maximal subfield $L$, then $F(\Ha)=L\cap\Sigma$ \cite[\S5, Cor.2]{abelianos}.

\begin{ex}\label{ex21}
When $\alge\cong\matrici_2(K)$, then $H(\Da)=J_X\cap\prod_{P\in|X|}\oink_P^*K_P^{*2}$, so that $\Sigma$ is the largest unramified exponent-2 abelian extension of $K$. When $X=\mathbb{P}_1(\finitum)$, so that $K=\finitum(t)$,
then $\Sigma=\mathbb{L}(t)$ for the unique quadratic extension $\mathbb{L}$ of $\finitum$.
\end{ex}

\paragraph{Proof of Theorem \ref{th1}:}
Let $U=X\backslash\{P\}$ be a maximal affine subset. The spinor class field $\Sigma_U$ of maximal
$\{P\}$-orders is the maximal subfield of $\Sigma$ splitting
completely at $P$. In particular,
$\Sigma_U=\Sigma$ if and only if $ P $ splits completely in
$\Sigma/K$. Otherwise, $\Sigma_U$ is a subextension with
$[\Sigma:\Sigma_U]=2$. If $ P $ in unramified for $\alge$, any two
maximal order $\Da$ and $\Da'$ are isomorphic on $U$ if and only
if their distance $\rho(\Da,\Da')$ is trivial on $\Sigma_U$. If this is the case, replacing $\Da'$ by a
(global) conjugate it can be assumed that $\Da(U)=\Da'(U)=D$. The
set of maximal orders satisfying the last relation is in
correspondence with the vertices of $T_P$. Two such orders are conjugate if and only if $\Da'=g\Da g^{-1}$
for some $g\in G$.

Let $e_P$ be an idele that is $1$ outside of $P$ and a
uniformizing parameter $\pi_P$ at $P$. Note that if $\Da$ and $\Da'$ are neighbors in $T_P$,
their completions $\Da_P$ and $\Da_P'$ have,  in some basis, the form
$$\Da_P=\sbmattrix {\oink_P}{\oink_P}{\oink_P}{\oink_P},\quad
\Da_P=\sbmattrix {\oink_P}{\pi_P^{-1}\oink_P}{\pi_P\oink_P}{\oink_P}=\sbmattrix100{\pi_P}\Da_P\sbmattrix100{\pi_P}^{-1}.$$ We conclude
that $\rho(\Da,\Da')=[e_P,\Sigma/K]=|[P,\Sigma/K]|$. It follows that the graph is bipartite whenever $|[P,\Sigma/K]|\neq\mathrm{id}_\Sigma$.
\qed

\section{Orders and vector bundles}\label{bundles}

In this section, notations are as in \S\ref{ndos}, except that we assume $\alge=\matrici_n(K)$.
In this case, any
maximal $X$-order on $\alge$ has the form $\Da=b\Da_0b^{-1}$ where
$b\in\alge_{\ad}$ is a matrix with adelic coefficients and
$\Da_0\cong\matrici_n(\oink_X)$. Note that the adelization is
$\Da_{0\ad}\cong\matrici_n(\oink_\ad)$, where $\oink_\ad\cong\prod_{P\in|X|}\oink_P$
is the ring of integral adeles (\S2).  In particular,  $\Da_{0\ad}^*$ is the 
group of adelic matrices $c$ satisfying $c\oink_X^n=\oink_X^n$. It
follows that $\Da_\ad^*$ is the group of all adelic matrices $c$
satisfying $c\Lambda=\Lambda$, where
$\Lambda=b\Lambda_0=b\oink_X^n$. Since the stabilizer of any
order $\Da_\wp$ in $\matrici_2(K_P)$  is $\Da_P^*K_P^*$, it follows that two
$X$-lattices $\Lambda_1$ and $\Lambda_2$ corresponds to the same
maximal order, if and only if $\Lambda_1=d\Lambda_2$ for some
 $d\in J_X$. Let $\mathrm{div}(d)$ be the divisor generated by $d$, i.e.,
 $\mathfrak{L}^{-\mathrm{div}(d)}=d\oink_X$. Note that every divisor is generated by an idele. Next result follows:

\begin{prop}
There is a correspondence between conjugacy classes of maximal
$X$-orders in $\matrici_n(K)$ and isomorphism classes of vector
bundles over $X$ up to multiplication by invertible bundles.
\end{prop}

Let $\Da_E=\mathcal{E}\!\textnormal{\footnotesize{$nd$}}_{\oink_X}(E)$ be
the maximal order corresponding to the vector bundle $E$. A finite algebra $\mathbb{B}$ acts globally as a ring of
endomorphisms of a vector bundle $E$ if and only if $\mathbb{B}$ embeds
into the ring of global sections $\Da_E(X)$. Note that the maximal
order $\Da_B$ defined in the introduction is the order $\Da_{E_B}$
corresponding to the bundle $E_B=\oink_X\oplus\mathfrak{L}^B$.
More generally, the maximal order corresponding to the bundle
$\mathfrak{L}^A\oplus\mathfrak{L}^B=\mathfrak{L}^A(\oink_X\oplus\mathfrak{L}^{B-A})$
is $\Da_{B-A}$. Note that a maximal $X$-order $\Da=\Da_E$ is split
if and only if any of the following equivalent conditions is
satisfied:
\begin{enumerate}
\item The algebra $\finitum^2=\finitum\times\finitum$ acts
globally on the vector bundle $E$. \item The algebra $\finitum^2$
embeds into the ring of global sections $\Da(X)$.  \item  The
commutative order $\Ha=\oink_X\times\oink_X$ embeds into
$\Da$.
\end{enumerate}
It follows from \cite[Cor.5.6]{abelianos} that every spinor
genera of maximal orders contain split orders. In fact, if $B=\mathrm{div}(b)$ is the
divisor generated by the idele $b$, then
$\Da_B=c\Da_0c^{-1}$ where $c=\sbmattrix 100b$. In particular, in
the notations of \cite[\S2]{abelianos}, the corresponding distance
element is $\rho(\Da_0,\Da_B)=[b,\Sigma/K]\in\mathrm{Gal}(\Sigma/K)$,
and therefore $\rho(\Da_A,\Da_B)=[a^{-1}b,\Sigma/K]\in\mathrm{Gal}(\Sigma/K)$.
By Example \ref{ex21} the spinor genera $\mathrm{Spin}(\Da_A)$ and $\mathrm{Spin}(\Da_B)$ coincide if and only if
$A-B\in2\mathrm{Pic}(X)$.

In general, if $\mathbb{B}\subseteq\matrici_2(K)$ is a finite $\finitum$-algebra, the dimension
 $\mathrm{dim}_{\finitum}\mathbb{B}$ can be arbitrarily large. However, we have next result:

\begin{prop} Assume $\mathbb{B}=\mathbb{B}'\oplus\mathbb{R}$ is a finite $\finitum$-algebra contained in
$\matrici_n(K)$, where $\mathbb{R}$ is the radical of $\mathbb{B}$. Then
$\mathrm{dim}_{\finitum}\mathbb{B}'=\mathrm{dim}_{K}(K\mathbb{B}')$, and the sum
$K\mathbb{B}'+K\mathbb{R}$ is direct.\end{prop}

\begin{proof}
 If  $\mathbb{B}=\bigoplus_{i=1}^nP_i\mathbb{B}$, where $P_1,\dots,P_n$ are the minimal central
idempotents of $\mathbb{B}$, then
$K\mathbb{B}=\bigoplus_{i=1}^nP_iK\mathbb{B}$. Therefore, we can assume that $\mathbb{B}'$ is simple. If
$\mathbb{B}'=\matrici_n(\mathbb{L})$ where $\mathbb{L}/\mathbb{F}$
is a finite extension, then $K\mathbb{B}'$ is a quotient of
$K\otimes_{\finitum}\mathbb{B}'\cong
\matrici_n(K\otimes_{\finitum}\mathbb{L})$. Since $\mathbb{F}$ is
the full field of constants of $K$, the tensor product
$K\otimes_{\finitum}\mathbb{L}$ is a field. It follows that
$K\otimes_{\finitum}\mathbb{B}'$ is simple and therefore equals
$K\mathbb{B}'$. The last statement follows since the two sided
ideal generated by an arbitrary non-invertible element $u$ in
$K\mathbb{B}'$ contains a non-trivial idempotent, and therefore
$u$ cannot belong to $K\mathbb{R}$.\end{proof}

\begin{cor} For any maximal order
$\Da$ in $\matrici_2(K)$, the semi-simple part of the ring
$\Da(X)$ is isomorphic to an element in the set
$\{\finitum,\finitum\times
\finitum,\mathbb{L},\matrici_2(\finitum)\}$, where $\mathbb{L}$ is
the unique quadratic extension of $\finitum$. Only the first two
cases a non-trivial radical $\mathbb{R}$  can exists, and in that case $\mathrm{dim}_K{K\mathbb{R}}=1$.\qed
\end{cor}
\begin{ex}\label{ex33}
 The bundles $E$ admitting an $\mathbb{L}$-vector space structure are those with $\Da_E(X)\cong\mathbb{L}$
or $\Da_E(X)\cong\matrici_2(\finitum)$. By the Matric Units Theorem \cite[p.30]{rowen}  we have
that $\Da(X)=\matrici_2(\finitum)$ implies
$\Da=\matrici_2(\oink_X)$. Moreover, $\mathbb{L}$ embeds into $\Da_E(X)$ if and only if
 $\Ha_{\mathbb{L}}=\mathbb{L}\otimes_{\finitum}\oink_X$ embeds into $\Da_E$.  Note that $\Ha_{\mathbb{L}}$ is the maximal order of $L=K\mathbb{L}$, and also the push-forward
sheaf $\Ha_{\mathbb{L}}=f_*(\oink_Y)$, where
$Y=X\times_{\mathrm{Spec}(\finitum)}\mathrm{Spec}(\mathbb{L})$ and $f:Y\rightarrow
X$ is the projection on the first coordinate.  The extension $L/K$ is unramified, whence $L\subseteq\Sigma$ (cf. Ex.\ref{ex21}). Let $\sigma$ be the generator of $\mathrm{Gal}(L/K)$. Then by the definition of the Artin map,
$|[B,L/K]|=\sigma^{\mathrm{deg}(B)}$ for any divisor $B$. Since $F(\Ha_{\mathbb{L}})=\Sigma\cap L=L$,
the order $\Ha_{\mathbb{L}}$
embeds in, precisely, the spinor genera $\Phi$
satisfying any of the following equivalent conditions:
\begin{enumerate}
\item For some (any) $\Da\in\Phi$, we have
$\rho(\Da_0,\Da)\big|_L=\mathrm{Id}_L$. \item For some (any) $\Da=c\Da_0c^{-1}\in\Phi$,
the integer $\mathrm{deg}\Big(\mathrm{div}\big(N(c)\big)\Big)$ is even. \item $\Phi$ contains an
order of the form $\Da_B$, where $B$ is a divisor of even degree.
\end{enumerate}
\end{ex}

\section{Split maximal orders}\label{split}

Next we study in greater detail the order $\Da_B$ defined in the
introduction. Note that if $B=D+\mathrm{div}(b)$, for any idele
$b$, then $\Da_B=c\Da_Dc^{-1}$, where $c=\sbmattrix b001$. It
follows that, when $B$ is a principal divisor, then
$\Da_B\cong\Da_0=\matrici_2(\oink_X)$, and therefore, the ring of
global sections $\Da_B(X)$ is isomorphic to the matrix algebra
$\matrici_2(\finitum)$. If $B$ is not principal, then
$\mathfrak{L}^B$ and $\mathfrak{L}^{-B}$ cannot have a global
section simultaneously. In fact, if $\mathrm{div}(f)+B\geq0$ and
$\mathrm{div}(g)-B\geq0$, then
$B=\mathrm{div}(g)=\mathrm{div}(f^{-1})$. We conclude that
$\Da_B(X)\cong (\finitum\times\finitum)\oplus V$, where the radical is $V=\mathfrak{L}^{\pm B}u$ with $u^2=0$.

\begin{prop}\label{class}
Assume $B$ is not principal. 
Any matrix $U$ satisfying  $\Da_B=U\Da_DU^{-1}$ has the form $\sbmattrix ab0c$,
in which case  $B$ is linearly equivalent to $D$, or  $\sbmattrix 0ac0$,
in which case  $B$ is linearly equivalent to $-D$.
\end{prop}

\begin{proof}
Let $U$ be as stated. Observe that, since $\Da_B\cong\Da_{-B}$ we can assume $\mathfrak{L}^{-B}(X)=\{0\}$.
If $W_B$ and $W_D$ denote the $K$-vector
spaces spanned by $\Da_B(X)$ and $\Da_D(X)$ respectively, then
$W_B=UW_DU^{-1}$.  Let $\{E_{i,j}\}_{i,j}$ be the cannonical basis
of the matrix algebra $\matrici_2(K)$. There are two cases two be
considered:
\begin{enumerate}
\item If $\mathfrak{L}^B(X)\neq\{0\}$, then
$W_B=W_D=KE_{1,1}\oplus KE_{2,2}\oplus KE_{1,2}$. \item If
$\mathfrak{L}^B(X)=\{0\}$, then $W_B=W_D=KE_{1,1}\oplus KE_{2,2}$.
\end{enumerate}
 In the first case, $U$ has the form
$\sbmattrix ab0c$. In particular we have
$$\mathfrak{L}^{-B}E_{2,1}=E_{2,2}\Da_BE_{1,1}=
E_{2,2}(U\Da_DU^{-1})E_{1,1}=a^{-1}c\mathfrak{L}^{-D}E_{2,1}.$$ We
conclude that $B=D+\mathrm{div}(ac^{-1})$, and therefore $B$ and
$D$ are linearly equivalent. In the second case $U$ has either the
form $\sbmattrix a00c$, which is similar to
the previous case, or the form $\sbmattrix 0ac0$, so that $B=-D+\mathrm{div}(ac^{-1})$, and $B$ is
linearly equivalent to $-D$.
\end{proof}

\paragraph{Proof of Theorem \ref{th2}:}
Let $e\in\Da(U)$ be a non-trivial idempotent and let $Z=Ke\oplus K(1-e)$ be the split semisimple commutative subalgebra generated by $e$. Let $\mathfrak{Z}\cong\oink_Xe\oplus\oink_X(1-e)$ be the unique maximal order in $Z$. By identifying the
vector space $K^2$ with $Z$, we see that the only local lattices that are invariant under $\mathfrak{Z}_P$ are
the fractional ideals $\big(\pi_P^re+\pi_P^s(1-e)\big)\mathfrak{Z}_P$, whence the corresponding maximal orders,
the ones containing  $\mathfrak{Z}_P$,  lie in a maximal path in the tree (or in the language of buildings, 
an apartment). We conclude that the
 maximal orders in that path are split, and moreover, this path has the form:
\[ { \xygraph{
!{<0cm,0cm>;<.8cm,0cm>:<0cm,.8cm>::}
 !{(2,0)}*+{\bullet}="a"  !{(2,0.4)}*+{{}^{\Da_{B-P}}}="a1"
 !{(4,0)}*+{\bullet}="b"   !{(4,0.4)}*+{{}^{\Da_{B}}}="b1"
!{(6,0) }*+{\bullet}="c"   !{(6,0.4)}*+{{}^{\Da_{B+P}}}="c1"
!{(0,0)}*+{}="e"
!{(8,0)}*+{}="d"
 "e"-@{.}"a" "a"-"b" "b"-"c"
"c"-@{.}"d"  } }.
\]
Recall from the classification of the split maximal orders given
earlier that two orders in this line can be conjugate if and only
if there exists different integers $N$ and $M$ such that the 
divisor classes $\bar B,\bar P\in\mathrm{Pic}(X)$ satisfy $\bar
B+N\bar P=\pm(\bar B+M\bar P)$. Since $P$ has positive degree,
only the equation with a negative sign can have non-trivial
solutions. In fact, this implies $2\bar B=(N+M)\bar P$. Replacing
$B$ by $B+kP$ if needed, we can assume $(N+M)\in\{0,1\}$, whence
either $2\bar B=0$ and $\bar B$ is an element of order $2$ in
$\mathrm{Pic}_0(X)$ or $2\bar B=P$, whence in the latter case the
place $P$ has even degree.\qed

\begin{rmk}
Note that, when $2\bar B=0$ or $2\bar B=P$, the image of this line
in the C-graph has, respectively, one of the following forms:
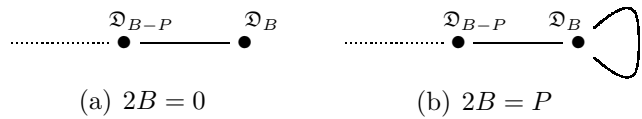
\begin{figure}[h]
\centering \mbox{\subfigure[$2B=0$]{\xygraph{
!{<0cm,0cm>;<.8cm,0cm>:<0cm,.8cm>::}
 !{(2,0)}*+{\bullet}="a"  !{(2.3,0.3)}*+{{}^{\Da_{B-P}}}="a1"
 !{(4,0)}*+{\bullet}="b"  !{(4.3,0.3)}*+{{}^{\Da_B}}="b1"
!{(4,-0.6)}*+{}="c"
 !{(0,0)}*+{}="e"
 "e"-@{.}"a" "a"-"b"
"b"-@{}"c" } } \quad \subfigure[$2B=P$]{\xygraph{
!{<0cm,0cm>;<.8cm,0cm>:<0cm,.8cm>::}
 !{(2,0)}*+{\bullet}="a"  !{(2.3,0.3)}*+{{}^{\Da_{B-P}}}="a1"
 !{(4,0)}*+{\bullet}="b"  !{(3.8,0.3)}*+{{}^{\Da_B}}="b1"
!{(0,0)}*+{}="e" "e"-@{.}"a" "a"-"b"
"b"-@`{p+(1,-1),p+(1,0),p+(1,1)}"b" } } }
\end{figure}

In the sequel, they will be called folded lines of type (a) or (b) respectively.
\end{rmk}

\section{Valencies in the S-graph}\label{vale}

In all of this section, let $\finitum=\finitum_p$ and $N=\mathrm{deg}(P)$, so that $\finitum(P)=\finitum_{p^N}$.
  Recall that the stabilizer in $\Gamma$ of a vertex $\Da$ is
the group invertible elements $\Da(X)^*$ of the ring of global
sections $\Da(X)$, and its action on the set of neighbors of $\Da$ can be realized identifying $\Da(X)^*$ with a subgroup of  $PGL_2\big(\finitum(P)\big)$, which acts naturally on the set of $\finitum(P)$-points of the projective line $\mathbb{P}_1(\finitum)$
\cite[\S II.1.1]{trees}. The number of orbits for all orders is given in Table 1. In this table,
$\epsilon$ is $0$ when $P$ has odd degree and $1$ otherwise.
\begin{table}
\begin{tabular}{ | c | c | c | c | c | }
  \hline
Type&$\Da(X)^*$ & number of orbits (valency)
 \\  \hline\hline
I&    $\matrici_2(\finitum)$ & $1+\frac{p^{N-1}-1}{p^2-1}+\frac{p}{p+1}\epsilon$ \\
\hline II&  $\finitum+\mathbb{R}$ & $p^{N-r}+1$
\\ \hline III&
 $(\finitum\times\finitum)+\mathbb{R}$ & $2+\frac{p^{N-r}-1}{p-1}$  \\  \hline
IV& $\mathbb{L}$ & $\frac{p^N+2p\epsilon+1}{p+1}$  \\
  \hline
\end{tabular}
\caption{Types of orders, and number of orbits in each case.}
\end{table}
To prove these values we compute, in each case, the number of elements in every conjugacy class of
$\Da(X)^*$, i.e., the number of matrices with any possible Jordan form. The number of invariant points
in each case is immediate (see Table 2), hence the result follows by an straightforward application of  Polya's formula  \textbf{(reference)}.
 \footnotesize
\begin{table}
\begin{tabular}{ | c | c | c | c | c | }
  \hline
{\scriptsize Jordan forms\footnotesize} & $\bmattrix b10b$ & $\bmattrix b00b$ & $\bmattrix b00c$ & NEV
 \\ \hline
{\scriptsize Fixed points\footnotesize} & $1$ & $p^N+1$ & $2$ & $2\epsilon$  \\
  \hline\hline
   I & $(p-1)^2(p+1)$ & $p-1$ & $\frac12(p^2-1)(p-2)p$ & $\frac12(p-1)^2p^2$  \\
\hline  II & $(p-1)(p^r-1)$ & $p-1$ & $0$ & $0$
\\ \hline
 III & $(p-1)(p^r-1)$ & $p-1$ & $(p-1)(p-2)p^r$ & $0$  \\  \hline
IV & $0$ & $p-1$ & $0$ & $(p-1)p$  \\
  \hline
\end{tabular}
\caption{Number of elements in $\Da(X)^*$ with every Jordan form for different types of orders.}
\end{table}
\normalsize In Table 2, NEV stands for \emph{no eigenvalues} on
$\finitum$. These elements have eigenvalues over the extension
$\finitum(P)$ if and only if $N=\deg P$ is even. We denote by $r$ the
dimension of the image of the radical $\mathbb{R}$ in the algebra
$\matrici_2\Big(\finitum(P)\Big)$. Certainly $r\leq N$. 

\begin{ex} \label{e41} We can have vertices of valency $1$ (or
endpoints) only if $N=1$ and in this case they are exactly the
maximal orders representing $\mathbb{L}\oink_X$ (compare to \cite[\S5]{Papikian}).
\end{ex}

\begin{ex} If $\alge$ is a division algebra, there are no
radicals, so in particular $r=0$. Furthermore, every vertex is in
case II or case IV. We conclude that a vertex
has valency $\frac{p^N+2p\epsilon+1}{p+1}$ if the corresponding
maximal order represents $\mathbb{L}\oink_X$ and $p^N+1$
otherwise (compare to \cite[\S5]{Papikian}).
\end{ex}

\begin{ex}\label{e44} When $\Da=\Da_B$ is split and $\mathrm{deg}(B)\geq0$, the neighbors corresponding to
$0$ is $\Da_{B-P}$, whence  $r$ is the dimension of
$\mathfrak{L}^B(X)/\mathfrak{L}^{B-P}(X)$, i.e. $r=l(B)-l(B-P)$ in
the notations of \cite[Ch.8]{Fulton}. In particular, when
$r=N$, the corresponding vertex has valency 2. By
Riemann-Roch's Theorem, this holds whenever
$\deg(B)\geq2g-2+N$ \cite[\S II.2.3, Lem.6]{trees}.
\end{ex}

A vertex in the C-graph $C_P$ is said to be unramified for the covering  $C_P\rightarrow S_P$
if it has the same valency than one (an therefore, every) point in its pre-image. For unramified vertices,
 the valencies in the C-graph are the ones we have already computed.

\begin{prop}\label{noram} A maximal order of the form $\Da_B$, where $B$ is a
divisor, is unramified,
unless $0$ is linearly equivalent to $2B$ but not $B$. 
\end{prop}

\begin{proof}
Assume first that $B$ is not principal.
Let $U$ be a global matrix satisfying $U\Da_BU^{-1}=\Da_B$.
By Proposition \ref{class}, we conclude
that $U=\sbmattrix ab0c$ or $U=\sbmattrix 0ac0$. In the first case
$ac^{-1}\mathfrak{L}^{-B}=\mathfrak{L}^{-B}$, and therefore
$ac^{-1}\in\finitum$. Replacing $U$ by a scalar multiple if
needed, we can assume $a,c\in\finitum^*$. Then comparing the first
coordinate of the identity $U\Da_B U^{-1}=\Da_B$, we obtain
$\oink_X+a^{-1}b\mathfrak{L}^{-B}=\oink_X$, and therefore
$b\mathfrak{L}^{-B}\subset\oink_X$. It follows that
$b\oink_X\subseteq\mathfrak{L}^B$, whence $b\in\mathfrak{L}^B(X)$.
We conclude that $U\in\Da(X)^*$. In the second case $B$ is
linearly equivalent to $-B$ by Proposition 
\ref{class}.

Assume now that $B$ is principal. Then any Global matrix $U$ satisfying $U\Da_BU^{-1}=\Da_B$ must, in particular,
satisfy $U\Da_B(X)U^{-1}=\Da_B(X)$. Since $\Da_B(X)$ is simple, every automorphism of it is inner. It follows that
$U\in K^*\Da_B(X)^*$.
\end{proof}

\begin{ex}\label{ex44} When  $2B=\mathrm{div}(f)$, we have $U\Da_BU^{-1}=\Da_B$
for $U=\sbmattrix 0f10$. In fact, by Proposition \ref{class}, any matrix  $U'=\sbmattrix 0ac0$
satisfying this condition is in the coset   $\sbmattrix 0f10K^*\Da_B(X)^*$.
\end{ex}

\begin{ex}\label{basic3}  Assume $X$ is the projective curve defined by the equation $y^2z=x^3+xz^2+z^3$, and let $P=[0;1;0]$ be the point at infinity. Then $\mathrm{Pic}(X)/\langle\bar P\rangle\cong\mathrm{Pic}_0(X)$ is a cyclic group with $4$ elements generated by the class of either $Q=[0;1;1]$ or $R=[0;-1;1]$. Then element of order $2$ is the class of $S=[1;0;1]$. The C-graph has three lines corresponding to the sets $\{\bar 0\}$,  $\{\bar S\}$,  $\{\bar Q, \bar R\}$. The first two lines are in the trivial component, since  $\bar 0$ and  $\bar S$ are squares. We conclude that the C-graph has the shape:
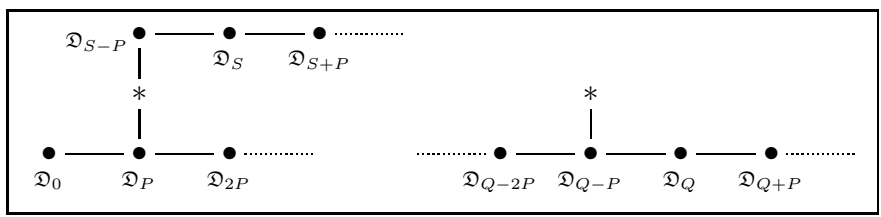
\[ \fbox{ \xygraph{
!{<0cm,0cm>;<.8cm,0cm>:<0cm,.8cm>::} 
!{(6,0.5)}*+{}="a" 
!{(7.5,0.5)}*+{\bullet}="c" !{(7.5,0)}*+{{}^{\Da_{Q-2P}}}="c1" 
!{(9,0.5)}*+{\bullet}="d" !{(9,0)}*+{{}^{\Da_{Q-P}}}="d1" 
!{(10.5,0.5)}*+{\bullet}="e" !{(10.5,0)}*+{{}^{\Da_{Q}}}="e1" 
!{(12,0.5)}*+{\bullet}="f" !{(12,0)}*+{{}^{\Da_{Q+P}}}="f1" 
!{(13.5,0.5)}*+{}="g" 
!{(9,1.5)}*+{*}="h" 
!{(0,0.5)}*+{\bullet}="i" !{(0,0)}*+{{}^{\Da_0}}="i1"
 !{(1.5,0.5)}*+{\bullet}="j"  !{(1.5,0)}*+{{}^{\Da_P}}="j1"
 !{(3,0.5)}*+{\bullet}="k"   !{(3,0)}*+{{}^{\Da_{2P}}}="k1"  
!{(4.5,0.5)}*+{}="l"
  !{(1.5,1.5)}*+{*}="m"
 !{(1.5,2.5)}*+{\bullet}="n"  !{(0.8,2.3)}*+{{}^{\Da_{S-P}}}="n1" 
 !{(3,2.5)}*+{\bullet}="o"  !{(3,2)}*+{{}^{\Da_{S}}}="o1"
  !{(4.5,2.5)}*+{\bullet}="p"   !{(4.5,2)}*+{{}^{\Da_{S+P}}}="p1"
 !{(6,2.5)}*+{}="q"  
"a"-@{.}"c" "c"-"d" "d"-"h" "e"-"d" "e"-"f" "f"-@{.}"g"
"p"-@{.}"q" "i"-"j" "j"-"k" "n"-"o" "o"-"p" "n"-"m" "j"-"m" "k"-@{.}"l"
} }
\]
The asterisques above represent unknown portions of the graph. The explicit description of the trivial component $S_P(\Da_0)$  given in \cite[p.87]{takahashi}, shows that $C_P(\Da_0)$ is as follows:
 \[ \xygraph{
!{<0cm,0cm>;<.8cm,0cm>:<0cm,.8cm>::} 
!{(1,0.5)}*+{\bullet}="i" !{(1,0)}*+{{}^{\Da_0}}="i1"
 !{(2.5,0.5)}*+{\bullet}="j"  !{(2.5,0)}*+{{}^{\Da_P}}="j1"
 !{(4,0.5)}*+{\bullet}="k"   !{(4,0)}*+{{}^{\Da_{2P}}}="k1"  
!{(5.5,0.5)}*+{}="l"
  !{(0.5,1.5)}*+{\bullet}="a"  !{(0,1.3)}*+{{}^{\mathfrak{V}_1}}="a1"
  !{(1.5,1.5)}*+{\bullet}="m"   !{(1.5,1)}*+{{}^{\mathfrak{O}}}="m1"
  !{(2.5,1.5)}*+{\bullet}="b"  !{(3,1.3)}*+{{}^{\mathfrak{V}_2}}="b1"
 !{(1.5,2.5)}*+{\bullet}="n"  !{(0.8,2.3)}*+{{}^{\Da_{S-P}}}="n1" 
 !{(3,2.5)}*+{\bullet}="o"  !{(3,2)}*+{{}^{\Da_{S}}}="o1"
  !{(4.5,2.5)}*+{\bullet}="p"   !{(4.5,2)}*+{{}^{\Da_{S+P}}}="p1"
 !{(6,2.5)}*+{}="q"  
"p"-@{.}"q" "i"-"j" "j"-"k" "n"-"o" "o"-"p" "n"-"m" "j"-"b" "a"-"m" "b"-"m" "k"-@{.}"l"
 }\]
In the notations of the reference, $\mathfrak{O}$ and $\mathfrak{V}_1$ are the images of the vertices
 $o$ and $v(-1)$,respectively, while $\mathfrak{V}_2$ is the image of both $v(1)$ and $v(\infty)$. The ramified
vertices are $\mathfrak{O}$ and $\Da_{S-P}$, the latter being the image of $v(0)$. The order $\mathfrak{O}$ is of type II, since $o$ has valency $4$ in the S-graph.
\end{ex}

\paragraph{Proof of Theorem 3} Assume first that $\mathrm{deg}(P)$ is odd. In particular $\Sigma_U=K$ (cf. \S \ref{ndos}), whence $C_P$ is connected.
We claim that $G=\Gamma$ in this case. Let $M$ be such that
 $M\Da_0(U)M^{-1}=\Da_0(U)$. The determinant of $M$ has even valuation at every place 
$Q\in U$, and therefore also on $P$,since principal divisors have degree $0$. Since the Picard group in this case is  $\mathrm{Pic}\Big(\mathbb{P}_1(\finitum)\Big)\cong \enteri$, $\mathrm{det}(M)$ is a square. We conclude that $\lambda M\in\Da_0(U)^*$ for some $\lambda\in K$, and the result follows.

Assume now that $N=\mathrm{deg}(P)$ is even. In this case $\Sigma_U=\Sigma$ is a quadratic extension (cf. Ex.\ref{ex21}), whence $C_P$ is connected. Then $P$ is linearly equivalent to some divisor of the form $2B$.
It follows that $\det(M)\in uK^{*2}\cup K^{*2}$ for any $u\in K^*$ with $\mathrm{div}(u)=P-2B$, whence we conclude
as before that $|G/\Gamma|\leq 2$. Since the line containing $\Da_B$ is a folded line of type (b), we must have equality.
Furthermore, a matrix whose determinant is not a square at $P$ cannot have an  invariant vertex since no vertices is ramified. We conclude that each orbit has two vertices.
\qed

\begin{ex}
When  $\alge=\matrici_2(K)$, $X\cong\mathbb{P}_1(\finitum)$, and $N=\deg P$ is odd, these graphs are the ones described in \cite[\S II.2.4]{trees}. When $N=1$, we have the graph described in \S1 by valency considerations alone.
In particular, no other vertices exists.
This proves Grothendieck-Birkhoff Theorem \cite[Th.2.1]{burban} in the particular case of two dimensional vector bundles over a finite field.
\end{ex}

\begin{ex}\label{basic2} The C-graphs for $N=2$ and $N=4$ are as follows:
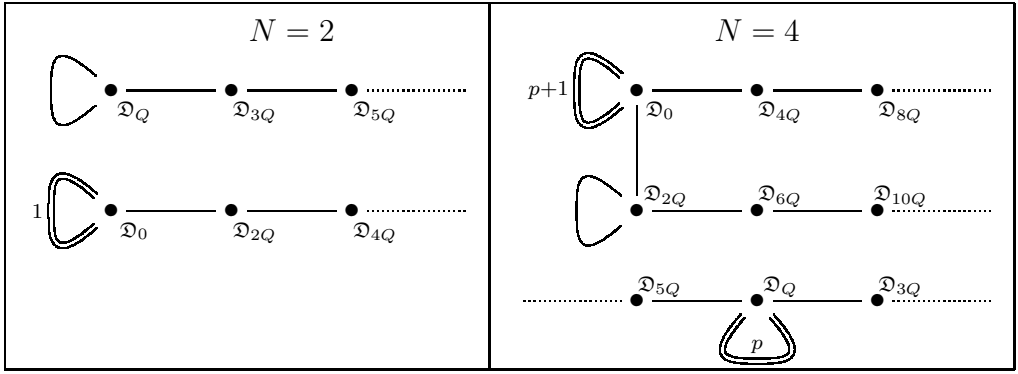
\[ \fbox{ \xygraph{
!{<0cm,0cm>;<.8cm,0cm>:<0cm,.8cm>::} 
!{(0,4)}*+{\bullet}="a" !{(0.4,3.6)}*+{{}^{\Da_Q}}="a1" 
!{(2,4)}*+{\bullet}="b" !{(2.4,3.6)}*+{{}^{\Da_{3Q}}}="b1" 
!{(4,4)}*+{\bullet}="c" !{(4.4,3.6)}*+{{}^{\Da_{5Q}}}="c1" 
!{(6,4)}*+{}="d" 
!{(0,2)}*+{\bullet}="e" !{(0.4,1.6)}*+{{}^{\Da_0}}="e1" 
!{(2,2)}*+{\bullet}="f" !{(2.4,1.6)}*+{{}^{\Da_{2Q}}}="f1" 
!{(4,2)}*+{\bullet}="g" !{(4.4,1.6)}*+{{}^{\Da_{4Q}}}="g1" 
!{(6,2)}*+{}="h"  
!{(3,5)}*+{N=2}="z" 
!{(1,-0.37)}*+{}="w"
"a"-@`{p+(-1,-1),p+(-1,0),p+(-1,1)}"a" "a"-"b" "b"-"c" "c"-@{.}"d"
"e"-@`{p+(-1,-1),p+(-1,0),p+(-1,1)}@{=}^1"e" "e"-"f" "f"-"g"
"g"-@{.}"h"} }
 \fbox{ \xygraph{
!{<0cm,0cm>;<.8cm,0cm>:<0cm,.8cm>::} 
!{(2,4)}*+{\bullet}="a"  !{(2.4,3.6)}*+{{}^{\Da_0}}="a1" 
!{(4,4)}*+{\bullet}="b" !{(4.4,3.6)}*+{{}^{\Da_{4Q}}}="b1" 
!{(6,4)}*+{\bullet}="c" !{(6.4,3.6)}*+{{}^{\Da_{8Q}}}="c1" 
!{(8,4)}*+{}="d" 
!{(2,2)}*+{\bullet}="e" !{(2.5,2.2)}*+{{}^{\Da_{2Q}}}="e1" 
!{(4,2)}*+{\bullet}="f" !{(4.4,2.2)}*+{{}^{\Da_{6Q}}}="f1" 
!{(6,2)}*+{\bullet}="g" !{(6.4,2.2)}*+{{}^{\Da_{10Q}}}="g1"  
!{(8,2)}*+{}="h" 
!{(2,0.5)}*+{\bullet}="i" !{(2.4,0.7)}*+{{}^{\Da_{5Q}}}="i1" 
!{(4,0.5)}*+{\bullet}="j" !{(4.4,0.7)}*+{{}^{\Da_Q}}="j1" 
!{(6,0.5)}*+{\bullet}="k" !{(6.4,0.7)}*+{{}^{\Da_{3Q}}}="k1"  
!{(8,0.5)}*+{}="l" 
!{(0,0.5)}*+{}="m"  
!{(4,5)}*+{N=4}="z"
"m"-@{.}"i" "i"-"j" "j"-"k" "k"-@{.}"l"
"j"-@`{p+(-1,-1),p+(0,-1),p+(1,-1)}@{=}^p"j" "a"-"e"
"a"-@`{p+(-1,-1),p+(-1,0),p+(-1,1)}@{=}^{p+1}"a" "a"-"b" "b"-"c"
"c"-@{.}"d" "e"-@`{p+(-1,-1),p+(-1,0),p+(-1,1)}"e" "e"-"f" "f"-"g"
"g"-@{.}"h" } } 
\]
The double lines are deduced by valency considerations. The line joining $\Da_{2Q}$ and $\Da_0$ for $N=4$ must be there since both lines are in the same connected component.
 When $N\geq6$, valencies alone do not suffice  to compute the multiplicity of all edges.
\end{ex}

\section{Multiplicity of edges}\label{multi}

 In this section, we show how the number of edges $M=M(\Da,\Da')$ can be computed in terms of 
the number of neighbors  $N=N(\Da,\Da')$ defined in $\S1$. As before, we limit ourselves to split edges.
Given a split maximal order $\Da=\Da_B$, the set of neighbors corresponding to a given edge in $S_P(\Da)$ is in correspondence with an orbit of $\Da(X)^*$ on  the $\finitum(P)$-points of the proyective line $\mathbb{P}^1(\finitum)$ (\S\ref{vale}). With this in mind, we divide the computation into three cases:

\subparagraph{Case A ($2\bar B\neq 0$): } Assume $\mathfrak{L}^{-B}(X)=\{0\}$, so that global sections are upper triangular. The orders in this case are unramified vertices for the cover  $S_P\rightarrow C_P$. In particular,  $C_P$ is locally homeomorphic to $S_P$, so it suffices to consider the  $\Da(X)^*$-action. On  $\mathbb{P}^1(\finitum)_{\finitum(P)}$, an element of the form  $\sbmattrix ab01$ acts as $t\mapsto at+b$. The only possible finite  solutions of $t=at+b$ , with $a\in\finitum$ ocur when $t$ is in the $r$-dimensional vector space of possible values of $b$ (\S\ref{vale}), and the latter form a single orbit, namely $[0]$. 
The class of $0$ corresponds to $\Da_{B-P}$ and the class of $\infty$ corresponds to $\Da_{B+P}$. 

\subparagraph{Case B ($2\bar B=0$, but $\bar B\neq 0$):}
These vertices have the same structure than the ones in case 1, except that they are ramified (cf. Ex.\ref{ex44}). In $C_P$, they are endpoints of folded lines of type (a). In this case the radical is $\mathbb{R}=\mathfrak{L}^B(X)E_{1,2}=0$. Conjugation by $F=\sbmattrix 0f10$ induces the map $x\mapsto f(P)/x$ on $\mathbb{P}^1(\finitum)_{\finitum(P)}$. The orbits under $\Da(X)^*$ have the form $[t]=\finitum^*t$. It follows that the invariant orbits $[\lambda]$ under conjugation by $F$ are given by the solutions of $f(P)/t=at$ for $a\in\finitum^*$, or $t^2=a^{-1}f(P)$. 

There are several sub-cases to be considered here.
\begin{itemize}
\item When the characteristic
$\mathrm{char}(\finitum)$ is $2$ or $\mathrm{deg}(P)$ is odd, there is always a unique invariant orbit corresponding to an order $\Da_1$.
In this case we can have $\tilde\Da_1\neq\Da_{B+P}$ (case \textbf{B1}), or  
$\tilde\Da_1=\Da_{B+P}$ (case \textbf{B2}).
\item When $\mathrm{char}(\finitum)\neq2$, $\mathrm{deg}(P)$ is even, and $f(P)$ is not a square in $\finitum(P)$,
there are no invariant orbits. This is case \textbf{B3}.
\item When $\mathrm{char}(\finitum)\neq2$, $\mathrm{deg}(P)$ is even, and $f(P)$ is a square in $\finitum(P)$,
there are two invarian orbits corresponding to $\Da_1$ and $\Da_2$. There are 4 diferent subcases. \begin{enumerate}
\item The orders
$\Da_1$ and $\Da_2$ can be isomorphic, in which case they can be isomorphic to 
 $\Da_{B+P}$  (case \textbf{B4}), or not (case \textbf{B5}).
\item The orders
$\Da_1$ and $\Da_2$ can fail to be isomorphic, in which case there can be one isomorphic to 
 $\Da_{B+P}$  (case \textbf{B6}), or none (case \textbf{B7}).
\end{enumerate}\end{itemize}

\subparagraph{Case C ($\bar B=0$):} 
In this case we can assume $B=0$. Then $\Da_B$ is an unramified vertex of the cover $\Gamma\backslash T\rightarrow G\backslash T$ (Prop.\ref{noram}). It suffices therefore to find the number of elements in each orbit of the usual action of $\Da_0(X)^*\cong \mathrm{PGL}_2(\finitum)$ on $\mathbb{P}^1(\finitum)_{\finitum(P)}$. Let $\mathbb{L}$ be the unique quadratic extension of $\finitum$. We know from the specific shape of the graphs for $N=1$ and $N=2$ (\S\ref{intro} and \S\ref{vale} respectively) that  $\mathrm{PGL}_2(\finitum)$ has one orbit on $\mathbb{P}^1(\finitum)_{\finitum}$ and two orbits on $\mathbb{P}^1(\finitum)_{\mathbb{L}}$. 
We  conclude that $\mathbb{P}^1(\finitum)_{\mathbb{L}}\backslash\mathbb{P}^1(\finitum)_{\finitum}$
is an orbit. Let $\hat\Da$ be the maximal order corresponding to this orbit.
 Since any equation of the type $x=(ax+b)/(cx+d)$ has all its roots in a quadratic extension, all elements outside $\mathbb{L}$ have trivial stabilizer. Let $\mu$ be such that $\mathbb{L}=\finitum(\mu)$.There are three subcases:
\begin{itemize}
\item If  $P$ has odd degree, $\mu$ is not an $\finitum(P)$-point of the projective plane. This is case \textbf{C1}.
\item If  $P$ has even degree,  the class $[\mu]$ corresponds to 
an order $\hat\Da$.  we can have $\hat\Da\cong\Da_{B+P}$ (case \textbf{C2}) or not  (case \textbf{C3}).
\end{itemize}

Table 3 covers the number of neighbors corresponding to every orbit in each case. Table 4 allows us to
compute the multiplicity of edges $M$ in terms of the number of neighbors $N$ in each case, assuming that we know
the identity of the exceptional orders $\Da_1$, $\Da_2$, or $\hat\Da$. The order $\Da_{B-P}$ is also considered exceptional to simplify the table. Certainly, $\Da_{B-P}\cong\Da_{B+P}$ except in case \textbf{A}.

\begin{table}
\begin{tabular}{|c||c|c|c|c|c|}
\hline
case&$[0]$&$[\infty]$&$[\lambda]$ & $[\mu]$&  other\\ \hline\hline
\textbf{A}&  $p^r$&$1$&-&-& $(p-1)p^r$\\ \hline
\textbf{B}& $2$&-& $(p-1)$&-& $2(p-1)$\\ \hline
\textbf{C}& $p+1$&-&-& $p(p-1)$& $p(p^2-1)$\\ \hline
\end{tabular}\caption{Wheights of edges.}\end{table}

\begin{table}
\begin{tabular}{|c||c|c|c|}
\hline
case&$\Da_{B+P}$& Exceptional if $\ncong\Da_{B+ P}$ & Other\\
\hline\hline
\textbf{A}&$1+\frac{N-1}{(p-1)p^r}$&$1+\frac{N-p^r}{(p-1)p^r}$&$\frac{N}{(p-1)p^r}$\\
\hline
\textbf{B1}&$2+\frac{N-(p-1)-2}{2(p-1)}$&-&$\frac{N}{2(p-1)}$\\
\hline
\textbf{B2}&$1+\frac{N-2}{2(p-1)}$&$1+\frac{N-(p-1)}{2(p-1)}$&$\frac{N}{2(p-1)}$\\
\hline
\textbf{B3}&$1+\frac{N-2}{2(p-1)}$&-&$\frac{N}{2(p-1)}$\\
\hline
\textbf{B4}&$2+\frac{N-2}{2(p-1)}$&-&$\frac{N}{2(p-1)}$\\
\hline
\textbf{B5}&$1+\frac{N-2}{2(p-1)}$&$1+\frac{N}{2(p-1)}$&$\frac{N}{2(p-1)}$\\
\hline
\textbf{B6}&$2+\frac{N-(p-1)-2}{2(p-1)}$&$1+\frac{N-(p-1)}{2(p-1)}$&$\frac{N}{2(p-1)}$\\
\hline
\textbf{B7}&$2+\frac{N-2}{2(p-1)}$&$1+\frac{N-(p-1)}{2(p-1)}$&$\frac{N}{2(p-1)}$\\
\hline
\textbf{C1}&$1+\frac{N-(p+1)}{p(p^2-1)}$&-&$\frac{N}{p(p^2-1)}$\\
\hline
\textbf{C2}&$1+\frac{N-(p+1)}{p(p^2-1)}$&$1+\frac{N-p(p-1)}{p(p^2-1)}$&$\frac{N}{p(p^2-1)}$\\
\hline
\textbf{C3}&$2+\frac{N-(p^2+1)}{p(p^2-1)}$&-&$\frac{N}{p(p^2-1)}$\\
\hline
\end{tabular}\caption{$M$ as a function of $N$.}\end{table}

Note that whenever $p>3$ we can tell this formulas apart by congruence conditions on $N$, 
except for cases \textbf{B4} and \textbf{B5}, where
the presence of two equal invariant orbits can be mistaken by a single orbit. In actual computations, it is preferable avoiding these vertices if at all possible.

\section{Representations.}\label{last}
In this section we show how explicit knowledge of the graphs can be used in the 
study of representation of orders and conversely.
Let $\Ha$ be a suborder of the maximal order $\Da$. In particular,
this implies that $\Ha(U)\subseteq\Da(U)$ for any affine open
subset $U\subseteq X$. On the other hand, if $\Ha$ is an order satisfying 
$\Ha(U)\subseteq\Da(U)$ for $U=X\backslash\{P\}$, then $\Ha\subseteq\Da$ if and only if  $\Ha_P\subseteq\Da_P$.
For any effective divisor $B$ we define the order $\Ha^{[B]}=\oink_X+\mathfrak{L}^{B}\Ha$.
Representation of orders relates to the Bruhat-Tits tree by the following fundamental result:

\begin{prop}\label{citedde}\cite[Prop.2.4]{Eichler2}.
Let $P$ be a prime divisor of a global function field $K=K(X)$.
 Let $\Ha$ be an arbitrary order in $\alge$.
Then $\Ha^{[tP]}$ is contained in a maximal order $\Da$ if and only if there exists a maximal order
$\Da'$ containing $\Ha$ such that the natural distance $\delta_P$ in the local Bruhat-Tits tree $T_P$
satisfies $\delta_P(\Da,\Da')\leq t$.
\end{prop}

\begin{ex}
Let $\Ha=\mathbb{L}\oink_X$ (cf. Ex. \ref{e41}), and let
$\Ha'=\Ha^{[3Q+4R+S]}$, where $Q,R,S$ are points
of degrees $1,2,4$ respectively. We define the intermediate orders
$\Ha''=\Ha^{[3Q]}$, and
$\Ha'''=\Ha^{[3Q+4R]}$. Recall that the only
maximal order containing a copy of $\Ha$ is $\Da_0$ since is the only vertex with valency 1 (cf. Ex. \ref{e41}). Then the
diagrams in \S1 and \S5 show that $\Ha''$ is contained
in $\Da_{tQ}$ for $t\leq3$, while $\Ha'''$ is contained in
$\Da_{tQ}$ for $t\leq11$, and finally $\Ha'$ is contained in
$\Da_{tQ}$ for $t\leq15$.
\end{ex}

\begin{ex}
Let  $X$ be the curve over $\finitum_2$ defined by the projective equation $x^2z+xz^2=y^3+yz^2+z^3$, and let $P=[1,0,0]$. Then $S_P(\Da_0)$ has the following structure \cite[\S II.2.4.4]{trees}:
\[ { \xygraph{
!{<0cm,0cm>;<.8cm,0cm>:<0cm,.8cm>::}  
!{(0,1.5)}*+{\bullet}="a" !{(0,1.8)}*+{{}^{x_0}}="a1" 
!{(2,1.5)}*+{\bullet}="b" !{(2,1.8)}*+{{}^{x_1}}="b1" 
 !{(4,1.5)}*+{\bullet}="c" !{(4,1.8)}*+{{}^{x_2}}="c1" 
!{(6,1.5)}*+{\bullet}="d" !{(6,1.8)}*+{{}^{x_3}}="d1" 
 !{(8,1.5)}*+{}="y" 
!{(4,0.5) }*+{\bullet}="e" !{(4,0)}*+{{}^{z_0}}="e1" 
!{(6,0.5)}*+{\bullet}="f"  !{(6,0)}*+{{}^{z_1}}="f1" 
!{(8,1) }*+{\bullet}="z"  !{(8.4,1)}*+{{}^{t_1}}="z1" 
!{(8,0) }*+{\bullet}="v"  !{(8.4,0)}*+{{}^{t_2}}="v1" 
"a"-"b" "b"-"c" "c"-"d" "d"-@{.}"y"
 "e"-"f" "f"-"z" "v"-"f"   "b"-"e"  }}
\]
The only vertices corresponding to split maximal orders are the $x_i$. It follows that the maximal order corresponding to
$t_1$ contains a copy of $\Ha^{[3P]}$, where $\Ha\cong\oink_X\times\oink_X$, but not  $\Ha^{[2P]}$.
\end{ex}

We say that $\Ha$ is optimally contained in $\Da$, if $\Ha\subseteq\Da$, but $\Ha$ is not contained in $\Da^{[B]}$ 
for any effective
divisor $B$.  Let $\Ha$ be an order of maximal rank, and let $\Da$ be a maximal order. The number of isomorphic copies of the order $\Ha$ optimally contained into the maximal order $\Da$ is denoted $I(\Ha|\Da)$.   The set of neighbors of $\Da$ is denoted $\mathcal{V}(\Da)$. We let $N_P(\Da,\Da')$ be as in \S\ref{intro}.

\begin{prop}
Let $\Ha$ be an order of maximal rank such that $\Ha_P$ is
maximal. Let $\Da$ be a maximal order containing $\Ha$. Then the
number $I(\Ha^{[P]}|\Da)$ is given by the formula:
$$I(\Ha^{[P]}|\Da)=\sum_{\Da'\in\mathcal{V}(\Da)}
N_P(\Da,\Da')I(\Ha|\Da').$$
\end{prop}

\begin{proof}
 It is immediate from Proposition \ref{citedde} that $\Ha^{[P]}$ is optimally contained in a maximal order $\Da$ if and only if there exists a maximal order $\Da'$ containing $\Ha$ such that the natural distance $\delta_P$ in the local Bruhat-Tits tree $T_P$ is exactly $1$.
Assume $\Ha^{[P]}\subseteq\Da$.  Since
$(\Ha^{[P]})_Q=\Ha_Q$ at every place $Q\neq P$, then $\Ha$ is contained in the order $\Da'$ defined by the local conditions:
$$\Da'_Q=\left\{\begin{array}{rcl}
\Da_Q&\textnormal{ if }&Q\neq P\\
\Ha_Q&\textnormal{ if }&Q=P\end{array}\right|.$$
This order is maximal and coincide with $\Da$ outside of $P$, whence it corresponds to a vertex in the Bruhat-Tits tree. On the other hand, $\Da'$ must be a neibour of $\Da$, since $\Da'_P=\Ha_P$ and $\Da_P$ contains $\Ha_P^{[P]}$.
Assume now that $\Da''$ is a second neighbour containing $\Ha$. Then $\Da''$ coincide with $\Da$, and therefore also with $\Da'$, outside of $P$. Furthermore $\Da''_P\supseteq\Ha_P=\Da'_P$, whence $\Da''=\Da'$.
We conclude that  $\Ha$ is contained in a unique $P$-neighbor of $\Da$
 and the result follows. 
\end{proof}

\paragraph{Proof of Theorem \ref{th3}.} 
It $\Ha$ is a maximal order, then it follows from previous proposition that 
$I(\Ha^{[P]}|\Da)=N_P(\Da,\Ha)$.
By a second application of the same result, for any pair of different places $(P_1,P_2)$, we have
$$I(\Ha^{[P_1+P_2]}|\Da)=\sum_{\Da'}N_{P_1}(\Da,\Da')N_{P_2}(\Da',\Ha),$$
where the sum extends over all maximal orders, but only a finite number of terms are non-zero. As the left hand side of this equation is symetric, the result folows.
\qed

\paragraph{Proof of Theorem \ref{th4}.} By an straightforward computation, both sides of (\ref{recurr}) coincide in every sufficiently high row or column. It follows that their difference is boundedly supported, and commutes with $N_1$. The result follows if we prove that the only finitely supported matrix that commutes with $N_1$ is the zero matrix.

Let $B=\sbmattrix A{O_1}{O_2}{O_3}$, where $A$ denotes a minimal finite block, and each $O_i$ 
is an infinite block of $0$'s.
Then, if $N_1=\sbmattrix CDEF$ is the analogous decomposition for $N_1$, the condition $BN_1=N_1B$ impplies $AD=0$
and $EA=0$. Looking at the first column of $AD$ and the first row of $EA$ we obtain the equations
$$\left(\begin{array}{cccc}0&\cdots&0&q\end{array}\right)A=
\left(\begin{array}{cccc}0&\cdots&0&0\end{array}\right),
\qquad A\left(\begin{array}{c}0\\ \vdots\\0\\1\end{array}\right)=
\left(\begin{array}{c}0\\ \vdots\\0\\0\end{array}\right),$$
where $q$ is either $p$ or $p+1$. This impplies that the last row and the last column of $A$ are $0$ 
and this contradicts the minimality of $A$. The result follows.
\qed

\begin{ex} The first few of the matrices $N_i$ are $$N_2=N_1^2-2pI,\qquad N_3=N_1^3-3pN_1,\qquad N_4=N_1^4-4N_1^2+2pI,$$ $$N_5=N_1^5-5pN_1^3+5pN_1,\qquad N_6=N_1^6-6pN_1^4,+9p^2N_1^2-2p^3I,$$
where $I$ denotes the identity matrix.
\end{ex}

\section{Examples}\label{exa}

\begin{ex} This example uses only Proposition \ref{citedde}.
Assume $X\cong\mathbb{P}_1(\mathbb{F})$ and $N=\deg P=5$.
 By valency considerations, and recalling that the
graph is bipartite in this case, we know the C-graph have the form:
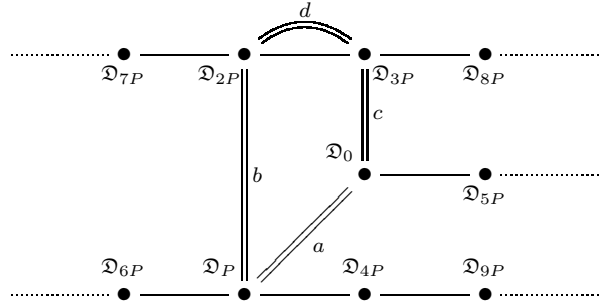
\[ { \xygraph{
!{<0cm,0cm>;<.8cm,0cm>:<0cm,.8cm>::} !{(0,4) }*+{}="x"
 !{(2,4)}*+{\bullet}="a"  !{(2,3.6)}*+{{}^{\Da_{7P}}}="a1" 
 !{(4,4)}*+{\bullet}="b"  !{(3.6,3.6)}*+{{}^{\Da_{2P}}}="b1" 
 !{(6,4)}*+{\bullet}="c"  !{(6.5,3.6)}*+{{}^{\Da_{3P}}}="c1" 
 !{(8,4)}*+{\bullet}="d"  !{(8,3.6)}*+{{}^{\Da_{8P}}}="d1" 
 !{(10,4)}*+{}="y"
 !{(6,2) }*+{\bullet}="e"  !{(5.6,2.3)}*+{{}^{\Da_0}}="e1" 
!{(8,2)}*+{\bullet}="f"   !{(8,1.6)}*+{{}^{\Da_{5P}}}="f1" 
 !{(10,2) }*+{}="z"
 !{(0,0) }*+{}="v"
!{(2,0) }*+{\bullet}="g"  !{(2,0.4)}*+{{}^{\Da_{6P}}}="g1" 
 !{(4,0)}*+{\bullet}="h"  !{(3.6,0.4)}*+{{}^{\Da_{P}}}="h1" 
!{(6,0) }*+{\bullet}="i"  !{(6,0.4)}*+{{}^{\Da_{4P}}}="i1" 
 !{(8,0)}*+{\bullet}="j"  !{(8,0.4)}*+{{}^{\Da_{9P}}}="j1" 
!{(10,0) }*+{}="w"
 "x"-@{.}"a" "a"-"b" "b"-"c" "c"-"d" "d"-@{.}"y"
 "e"-"f" "f"-@{.}"z" "v"-@{.}"g" "g"-"h" "h"-"i" "i"-"j" "j"-@{.}"w"
"b"-@{=}^b"h" "c"-@{=}^c"e" "e"-@{=}^a"h" "b"-@{=}@`{p+(-1,1)}^d"c"}}
\]
Where the valency table gives us the following system:
$$c+d=1,\quad c+a=p^2+1,\quad a+b=p^2+p+1,\quad b+d=p+1.$$
It follows that either $(a,b,c,d)=(p^2,p+1,1,0)$ or
$(a,b,c,d)=(p^2+1,p,0,1)$.
Assume the second solution. Let $\Ha=\Da_{7Q}^{[P+R]}$,
 where $R$ has degree $4$. Deffine also $\Ha'=\Da_{7Q}^{[R]}$
and $\Ha''=\Da_{7Q}^{[P]}$. One application of Prop. \ref{citedde}
shows that the set of
 maximal orders containing $\Ha'$ is $\{\Da_{7Q},\Da_{3Q},\Da_{11Q}\}$. Now, a 
second aplication of Prop. \ref{citedde} shows that the set of
 maximal orders containing $\Ha'$ is $$A=\{\Da_{7Q},\Da_{3Q},\Da_{11Q},\Da_{2Q},
\Da_{12Q},\Da_{8Q},\Da_{6Q},\Da_{16Q}\}.$$ If we use $\Ha''$ instead, we obtain the set
 $A\cup\{\Da_0\}$.
As these sets are different, we conclude that the second
solution is inconsistent, and therefore $(a,b,c,d)=(p^2,p+1,1,0)$.
\end{ex}

\begin{ex} Let $X$ be as in Example \ref{basic3}, and let $Z$ be the prime divisor corresponding to the point
 $[\alpha;\alpha;1]\in X_{\finitum_{27}}$. Then $\mathrm{div}(x-y)=Z-3P$, so $Z$ is linearly equivalent to $3P$.
Two lines in the local graph at $Z$ are as follows:
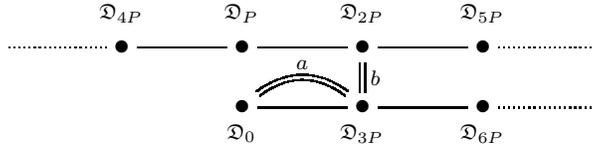
\[ \xygraph{
!{<0cm,0cm>;<.8cm,0cm>:<0cm,.8cm>::} 
!{(0,2) }*+{}="a" 
!{(2,2)}*+{\bullet}="b" !{(2,2.5)}*+{{}^{\Da_{4P}}}="b1" 
!{(4,2)}*+{\bullet}="c" !{(4,2.5)}*+{{}^{\Da_{P}}}="c1" 
!{(6,2)}*+{\bullet}="d" !{(6,2.5)}*+{{}^{\Da_{2P}}}="d1" 
!{(8,2)}*+{\bullet}="e" !{(8,2.5)}*+{{}^{\Da_{5P}}}="e1" 
 !{(10,2) }*+{}="f" 
!{(4,1)}*+{\bullet}="g" !{(4,0.5)}*+{{}^{\Da_{0}}}="g1" 
!{(6,1)}*+{\bullet}="h" !{(6,0.5)}*+{{}^{\Da_{3P}}}="h1" 
!{(8,1)}*+{\bullet}="i" !{(8,0.5)}*+{{}^{\Da_{6P}}}="i1" 
 !{(10,1) }*+{}="j" 
"a"-@{.}"b" "e"-@{.}"f" "c"-"b" "c"-"d" "d"-"e"
"i"-@{.}"j" "g"-"h" "h"-"i" "g"-@{=}@`{p+(-1,1)}^a"h" "d"-@{=}^b"h" 
} .
\] 
The double lines denote two of the many possible additional edges. We compute their multiplicity.
First we compute $a$. We use the equation
$$N_P(\Da_0,\Da_P)N_Z(\Da_P,\Da_{4P})=N_Z(\Da_0,\Da_{3P})N_P(\Da_{3P},\Da_{4P}),$$
corresponding to the pair $(\Da_0,\Da_{4P})$. Note that all other terms are $0$, since $\Da_{5P}$ has valency $2$
at $Z$, while neither
$\Da_0$ nor $\Da_{4P}$ has other neighbors at $P$.  This equation gives  
$N_Z(\Da_0,\Da_{3P})=4$, since all other values are known. For instance,
$N_Z(\Da_P,\Da_{4P})=1$ since $\Da_{4P}$ corresponds to the orbit $[\infty]$ at $\Da_P$. 
 Now, since $4$ is the number of elements
in the orbit $[0]$ in case \textbf{C} with $p=3$, this is the only edge connecting $\Da_0$ with $\Da_{3P}$.

Next we show $b=0$. Assume $\Da_{2P}$ and $\Da_{3P}$ are neighbors. Then the equation 
corresponding to the pair $(\Da_{3P},\Da_{3P})$ reduces to 
$$N_P(\Da_{3P},\Da_{2P})N_Z(\Da_{2P},\Da_{3P})=N_Z(\Da_{3P},\Da_{2P})N_P(\Da_{2P},\Da_{3P}),$$ 
since $\Da_{4P}$ and $\Da_{3P}$ are not neighbors at $Z$. The extra edge for either
$\Da_{2P}$ or $\Da_{3P}$ at $Z$ corresponds to an orbit of size $18$, whence the equation gives
$3\times18=18\times1$. The contradiction yields the conclusion.
\end{ex}

\begin{ex}
When $X=\mathbb{P}_1(\finitum)$ and $\mathrm{deg}(P)=6$, the C-graph is as follows:
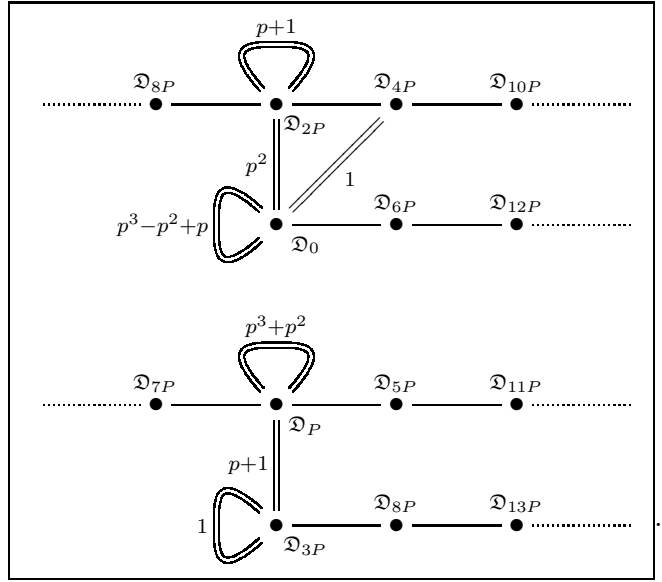
\[ \fbox{ \xygraph{
!{<0cm,0cm>;<.8cm,0cm>:<0cm,.8cm>::}
 !{(0,2) }*+{}="j1"
 !{(2,2)}*+{\bullet}="a1"  !{(2,2.3)}*+{{}^{\Da_{7P}}}="a11"
 !{(4,2) }*+{\bullet}="b1"  !{(4.5,1.6)}*+{{}^{\Da_{P}}}="b11"
!{(6,2)}*+{\bullet}="c1"  !{(6,2.3)}*+{{}^{\Da_{5P}}}="c11"
 !{(8,2) }*+{\bullet}="d1"  !{(8,2.3)}*+{{}^{\Da_{11P}}}="d11"
!{(10,2)}*+{}="e1"
 !{(4,0) }*+{\bullet}="f1"  !{(4.5,-0.4)}*+{{}^{\Da_{3P}}}="f11"
!{(6,0)}*+{\bullet}="g1"  !{(6,0.3)}*+{{}^{\Da_{8P}}}="g11"
 !{(8,0) }*+{\bullet}="h1"  !{(8,0.3)}*+{{}^{\Da_{13P}}}="h11"
!{(10,0)}*+{}="i1" 
 !{(0,7) }*+{}="j2"
 !{(2,7)}*+{\bullet}="a2"  !{(2,7.3)}*+{{}^{\Da_{8P}}}="a21"
 !{(4,7) }*+{\bullet}="b2"  !{(4.5,6.6)}*+{{}^{\Da_{2P}}}="b21"
!{(6,7)}*+{\bullet}="c2"  !{(6,7.3)}*+{{}^{\Da_{4P}}}="c21"
 !{(8,7) }*+{\bullet}="d2"  !{(8,7.3)}*+{{}^{\Da_{10P}}}="d21"
!{(10,7)}*+{}="e2"
 !{(4,5) }*+{\bullet}="f2"  !{(4.5,4.6)}*+{{}^{\Da_0}}="f21"
!{(6,5)}*+{\bullet}="g2"  !{(6,5.3)}*+{{}^{\Da_{6P}}}="g21"
 !{(8,5) }*+{\bullet}="h2"  !{(8,5.3)}*+{{}^{\Da_{12P}}}="h21"
!{(10,5)}*+{}="i2"
"b1"-@`{p+(-1,1),p+(0,1),p+(1,1)}@{=}^{p^3+p^2}"b1"
"a1"-"b1" "b1"-"c1" "c1"-"d1" "d1"-@{.}"e1" "a1"-@{.}"j1"
"f1"-@`{p+(-1,-1),p+(-1,0),p+(-1,1)}@{=}^1"f1" "f1"-@{=}^{p+1}"b1" "f1"-"g1"
"g1"-"h1" "h1"-@{.}"i1" 
"b2"-@`{p+(-1,1),p+(0,1),p+(1,1)}@{=}^{p+1}"b2"
"a2"-"b2" "b2"-"c2" "c2"-"d2" "d2"-@{.}"e2" "a2"-@{.}"j2"
"f2"-@`{p+(-1,-1),p+(-1,0),p+(-1,1)}@{=}^{p^3-p^2+p}"f2" "f2"-@{=}^{p^2}"b2" "f2"-@{=}_1"c2" "f2"-"g2"
"g2"-"h2" "h2"-@{.}"i2" } }.
\]
The multiplicities can be computed one by one, as in the preceding example, or alternatively, we can use the matrix
\scriptsize
$$ N_6= \left(\begin{array}{cccccccccc}
p^5(p-1)&0&p^5(p-1)&0&p^5(p-1)&0&p^6&0&\cdots\\
0&p^4(p^2-1)&0&p^4(p^2-1)&0&p^6&0&p^6&\cdots\\
p^3(p^2-1)&0&p^3(p^2-1)&0&p^5&0&0&0&\cdots\\
0&p^2(p^2-1)&0&p^4&0&0&0&0&\cdots\\
p(p^2-1)&0&p^3&0&0&0&0&0&\cdots\\
0&p^2&0&0&0&0&0&0&\cdots\\
p+1&0&0&0&0&0&0&0&\cdots\\
0&1&0&0&0&0&0&0&\cdots\\
\vdots&\vdots&\vdots&\vdots&\vdots&\vdots&\vdots&\vdots&\ddots
\end{array}\right).$$ \normalsize
 Note that $M(\Da_0,\Da_0)\equiv p(p-1)\ \big(\mathrm{mod}\ p(p^2-1)\big)$, whence the exceptional edge for $\Da_0$ is in the loop.

\end{ex}


\begin{thebibliography}{99}

\bibitem{spinor}
{\scshape L.E. Arenas-Carmona}, Applications of spinor class
fields: embeddings of orders and quaternionic lattices,
\textit{Ann. Inst. Fourier} \textbf{53} (2003), 2021-2038.



\bibitem{Eichler}
{\scshape L.E. Arenas-Carmona}, \textit{Relative spinor class fields: A counterexample}, Archiv.
Math. \textbf{91} (2008), 486-491..


\bibitem{abelianos}
{\scshape L.E. Arenas-Carmona}, \textit{Representation fields for
commutative orders}, to appear in Ann. Inst. Fourier.
arXiv:1104.1809v1 [math.NT].

\bibitem{Eichler2}
{\scshape L.E. Arenas-Carmona}, \textit{Trees, branches, and
spinor genera}, Preprint. arXiv:1111.1473v1 [math.NT].

\bibitem{ab2}
{\scshape L.E. Arenas-Carmona}, \textit{Representation fields for cyclic orders}, To appear in Acta Arith.


\bibitem{burban}
{\scshape L.Bodnarchuk, I.Burban, Yu.Drozd, G.-M.Greuel}, Vector
bundles and torsion free sheaves on degenerations of elliptic
curves, in \textit{Global Aspects of Complex Geometry}, 83-129,
Springer, (2006), \textit{math.AG/0603261}.

\bibitem{brzezinski87}
{\scshape J. Brzezinski}, Riemann-Roch Theorem for locally
principal orders, \textit{Math. Ann.}, 276 (1987), 529-536.

%








\bibitem{Fulton}
{\scshape W. Fulton}, \textit{Algebraic curves, an introduction to algebraic
geometry}, Addison-Wesley, Redwood City, 1989.





\bibitem{Mason1}
{\scshape  A. W. Mason}, Serre's generalization of Nagao's theorem: an elementary approach. \textit{Trans. Amer. Math. Soc.} \textbf{353.2} (2001), 749-767.

\bibitem{Mason2}
{\scshape  A. W. Mason}, The generalization of Nagao's theorem to other subrings of the rational function field.  \textit{Comm. Algebra} \textbf{31.11} (2003), 5199-5242.

\bibitem{Mason3}
{\scshape  A. W. Mason and A. Schweizer}, The minimun index of a non-congruence subgroup of $\mathrm{SL}_2$ over an arithmetic domain.  \textit{Israel J. Math.} \textbf{133} (2003), 29-44.

\bibitem{Mason4}
{\scshape  A. W. Mason and A. Schweizer}, The minimun index of a non-congruence subgroup of $\mathrm{SL}_2$ over an arithmetic domain II. The rank zero cases.  \textit{ J. London Math. Soc. (2)} \textbf{71} (2005), 53-68.

\bibitem{Mason5}
{\scshape A.W. Mason and A. Schweizer}, Nonrational genus zero function
fields and the Bruhat-Tits tree,  \textit{Comm.
Algebra}, \textbf{37.12} (2009), 4241-4258.



\bibitem{Papikian}
{\scshape M. Papikian}, Local diophantine properties of modular curves of
D-elliptic sheaves, \textit{J. reine und angew. Math}, to appear.
url:http://www.math.psu.edu/papikian/Research/research.html.


\bibitem{rowen}
{\scshape L.H. Rowen}, \textit{Ring Theory}, \textbf{Vol. I}, Academic Press, Boston, 1988.




\bibitem{takahashi}
{\scshape S. Takahashi},  \textit{The fundamental domain of the tree
of $GL(2)$ over the function field of an elliptic curve}, Duke Math. J. \textbf{72.1} (1993), 85-97.

\bibitem{trees}
{\scshape J.-P. Serre},  \textit{Trees}, Springer Verlag, Berlin,
1980.


\bibitem{weil}
{\scshape A. Weil.} \textit{Basic Number Theory}, $2^{\mathrm{nd}}$ Ed.,
 Springer Verlag, Berlin, 1973.
\end{thebibliography}
\end{document}